\documentclass[10pt]{article}%
\usepackage[final]{graphics}
\usepackage{graphicx}
\usepackage{bm}
\usepackage{amsmath}
\usepackage{amssymb}
\usepackage{amsfonts}
\usepackage{verbatim}
\usepackage{amsthm}
\usepackage{mathtools}
\usepackage{relsize}
\usepackage{hyperref}
\usepackage{xcolor}%
\providecommand{\U}[1]{\protect\rule{.1in}{.1in}}
\numberwithin{equation}{section}	
\theoremstyle{plain}
\newtheorem{theorem}{Theorem}[section]

\newtheorem{lemma}[theorem]{Lemma}

\newtheorem{corollary}[theorem]{Corollary}

\theoremstyle{definition}

\newtheorem{remark}[theorem]{Remark}
\input epsf
\def\protectbold#1{\protect{\boldmath{$#1$}}}

\def\bigO{{\cal O}}

\def\dsp#1{\displaystyle#1}
\def\protectbold#1{\protect{\boldmath{$#1$}}}

\begin{document}

\title{Asymptotic expansions of Kummer hypergeometric functions with three asymptotic
parameters $a$, $b$ and $z$}
\author{N. M. Temme\thanks{IAA, 1825 BD 25, Alkmaar, The Netherlands. Former address:
Centrum Wiskunde \& Informatica (CWI), Science Park 123, 1098 XG Amsterdam,
The Netherlands. Email: nico.temme@cwi.nl} \quad\quad E. J. M.
Veling\thanks{Noord-Houdringelaan 22, 3722 BR Bilthoven, The Netherlands.
Former address: Delft University of Technology (TUDelft), Faculty of Civil
Engineering and Geosciences, Water Resources Section, Delft, The Netherlands.
Email: ed.veling@xs4all.nl} }
\maketitle

\begin{abstract}
\noindent In a recent paper \cite{Temme:2021:AKH} new asymptotic expansions
are given for the Kummer functions $M(a,b,z)$ and $U(a,b+1,z)$ for large
positive values of $a$ and $b$, with $z$ fixed and special attention for the
case $a\sim b$. In this paper we extend the approach and also accept large
values of $z$. The new expansions are valid when at least one of the parameters $a$, $b$, or $z$ is large.
We provide numerical tables to show the performance of the expansions.
\end{abstract}


\date{\ }

{\small \noindent\textbf{Keywords} Asymptotic expansions; Kummer functions;
Confluent hypergeometric functions\newline\textbf{AMS Classification} Primary
41A60; Secondary 33C15}

\section{Introduction}

\label{sec:intro}

We derive new asymptotic expansions of the Kummer functions $M(a,b,z)$ and
$U(a,b+1,z)$ in which all three positive parameters $a$, $b$, and $z$ are
allowed to be large, and they are even valid when at least one of the parameters $a$, $b$, or $z$ is large. 
The methods of the recent paper \cite{Temme:2021:AKH}
require only a minor modification to include the argument $z$ as a large
parameter. Again we use a uniform method to derive the asymptotic expansion of
a Laplace-type integral of the form
\begin{equation}
F_{\lambda}(z)=\frac{1}{\Gamma(\lambda)}\int_{0}^{\infty}s^{\lambda-1}%
e^{-zs}f(s)\,ds, \label{eq:int01}%
\end{equation}
with $z$ as a large positive parameter. If the parameter $\lambda>0$ is fixed
we can use Watson's lemma. However, when $\lambda$ is allowed to become large,
there is a positive saddle point, and the asymptotic approach can be based on
Laplace's method. The uniformity aspect is that we combine both methods in one approach.

The method can also be used for loop integrals of the form
\begin{equation}
\label{eq:int02}G_{\lambda}(z)=\frac{\Gamma(\lambda+1)}{2\pi i}\int_{-\infty
}^{(0+)} s^{-\lambda-1}e^{zs}g(s)\,ds,
\end{equation}
where the contour runs from $-\infty$ with $\mathrm{ph}\,s=-\pi$, encircles
the origin in anti-clockwise direction, and returns to $-\infty$ with
$\mathrm{ph}\,s=\pi$. The negative axis is a branch cut and we assume that
$s^{-\lambda-1}$ has real values for $s>0$ (when $\lambda$ is real).

We show in the next section that $G_{\lambda}(z)$ can be interpreted as an
analytic continuation with respect to $\lambda$ of $F_{\lambda}(z)$ (after changing a notation), and in
this way we can reduce the number of four expansions needed in
\cite{Temme:2021:AKH} to only two.

The asymptotic analysis of Kummer functions (or confluent hypergeometric
functions) has been discussed in great detail in the literature. A simple
result is available for $M(a,b,z)$ when $b\to\infty$, with $a=\mathcal{O}(1)$
and $z=\mathcal{O}(1)$, because in that case the defining convergent power
series has an asymptotic character. In \cite[Chapter~10]{Temme:2015:AMI}
several results for $M(a,b,z)$ and $U(a,b,z)$ are derived for large $a$ or
$b$, also in combination with large $z$. The classical results on large $z$
expansions are considered in \cite{Olver:1997:ASF} and \cite{Slater:1960:CHF},
and summarised in \cite{Olde:2010:CHF}, where we can also find expansions for
large parameters. See also \cite{Dunster:1989:UAE} and \cite{Olver:1980:UAE},
where the results are derived for the Whittaker functions. In the notation of
the Whittaker functions $M_{\kappa,\mu}(z)$ and $W_{\kappa,\mu}(z)$, the
uniformity aspects considered in the present paper are similar to those with
large $z$, $\kappa$ and $\mu$, paying special attention to the case
$\kappa\sim-\mu$, $\mu>0$. In a recent paper \cite{Dunster:2021:UAE} expansions are given for the Whittaker functions
for large values of $\mu$, which are uniformly valid for $0\le \kappa/\mu\le 1-\delta$ and $0\le {\rm arg}(z)\le\pi$.
Large values of $\mu$ correspond to large $a$ and $b$ such that $b\sim 2a$.

\section{The asymptotic method}

\label{sec:method} We summarise the main steps in the construction of the
asymptotic expansions. For details we refer to the Appendix and  \cite[Chapter~25]%
{Temme:2015:AMI}.

The Kummer functions can be written as integrals of the form
\begin{equation}
\int_{0}^{c}e^{-z\phi(t)}\Phi(t)\,dt, \label{eq:met01}%
\end{equation}
where $c=1$ for the $M$-function and $c=\infty$ for the $U-$function. The
functions $\Phi(t)$ and $\phi(t)$ will be described in a later section, see \eqref{eq:Mbga04}.
The function $\phi(t)$ has one saddle point $t_{0}$ (a zero of $\phi^{\prime}(t)$) in $(0,c)$, and
$t_{0}$ moves to the origin under the influence of an extra parameter. A
typical example is the function
\begin{equation}
\psi(s)=s-\mu\ln s,\quad\psi^{\prime}(s)=\frac{s-\mu}{s}, \label{eq:met02}%
\end{equation}
where the saddle point $s_{0}=\mu$ tends to zero when $\mu$ does. At that same
time $\psi(s)\rightarrow s$, and the saddle point vanishes. In
\cite[Chapter~25]{Temme:2015:AMI} this asymptotic feature has been called
\emph{The vanishing saddle point}. The asymptotic method is based on the
transformation
\begin{equation}
\phi(t)-\phi(t_{0})=\psi(s)-\psi(s_{0}),\quad\psi(s)=s-\mu\ln s,
\label{eq:met03}%
\end{equation}
with condition $\mathrm{sign}(s-s_{0})=\mathrm{sign}(t-t_{0})$.

In the cases covered in this paper,
the  integral as in \eqref{eq:met01} will become a
Laplace-type integral as in \eqref{eq:int01}:
\begin{equation}
e^{-z(\phi(t_{0})-\psi(s_{0}))}\int_{0}^{\infty}s^{\lambda-1}e^{-zs}%
f(s)\,ds,\quad\lambda=\mu z,\quad f(s)=s\Phi(t)\frac{dt}{ds}.
\label{eq:met04}%
\end{equation}

Using an integration by parts scheme we can obtain the asymptotic expansion of
\eqref{eq:int01}
\begin{equation}
\label{eq:met05}F_{\lambda}(z)\sim z^{-\lambda}\sum_{n=0}^{\infty}\frac
{f_{n}(\mu)}{z^{n}},\quad z\to\infty,\quad\lambda\ge0.
\end{equation}
In the Appendix we explain how the coefficients $f_{n}(\mu)$ can be obtained.
In the same way we can find the expansion of the contour integral in
\eqref{eq:int02}
\begin{equation}
\label{eq:met06}G_{\lambda}(z)\sim z^{\lambda}\sum_{n=0}^{\infty}(-1)^{n}%
\frac{g_{n}(\mu)}{z^{n}},\quad z\to\infty, \quad\lambda\ge0.
\end{equation}

In the previous article \cite{Temme:2021:AKH} we considered the cases $a\le b$
and $a\ge b$ for each Kummer function, resulting in four expansions. Here we
combine the method for $a\le b$ and $a\ge b$ by exploiting the strong
relationship between the functions $F_{\lambda}(z)$ and $G_{\lambda}(z)$,
because $G_{\lambda}(z)$ can be seen as the analytic continuation with respect to $\lambda$ of
$F_{\lambda}(z)$, which becomes defined for $\Re\lambda\le0$.

To show this, we first observe that the integral in \eqref{eq:int02} exists
for all finite complex values of $\lambda\ne-1,-2,,3,\ldots$. For a start, let
$\Re\lambda<0$. Then we can use $(-\infty,0]$ as the path of integration and
obtain
\begin{equation}
\label{eq:met07}%
\begin{array}
[c]{@{}r@{\;}c@{\;}l@{}}%
G_{\lambda}(z) & = & \displaystyle \frac{\Gamma(\lambda+1)}{2\pi i}\left(
e^{-\pi i(-\lambda-1)}-e^{\pi i(-\lambda-1)}\right)  \int_{-\infty}^{0} \vert
s\vert^{-\lambda-1}e^{zs} g(s)\,ds\\
& = & \displaystyle \frac{1}{\Gamma(-\lambda)}\int_{0}^{\infty}s^{-\lambda
-1}e^{-zs} g(-s)\,ds.
\end{array}
\end{equation}
This can be used for $\lambda= -1,-2,,3,\ldots$ and it becomes $F_{\lambda
}(z)$ when we replace $\lambda$ by $-\lambda$ and $g(-s)$ by $f(s)$. We
conclude this as follows.

\begin{lemma}
\label{lem:lem01} The function $G_{\lambda}(z)$ defined in \eqref{eq:int02} is
an analytic function for all complex values of $\lambda$ and is, with
$g(s)=f(-s)$, the analytic continuation with respect to $\lambda$  of $F_{\lambda}(z)$, which is
initially defined for $\Re\lambda>0$.
\end{lemma}

In the following sections we use this lemma when an asymptotic expansion
derived for $b\ge a$ will also be used for $b\le a$. In this way we reduce the
four different methods used in \cite{Temme:2021:AKH} to two approaches.

In this paper we use the positive argument $z$ of the Kummer functions as the
principal asymptotic parameter, in the asymptotic analysis we scale the
parameters $a$ and $b$ with respect to $z$ by using
\begin{equation}
\alpha=\frac{a}{z},\quad\beta=\frac{b}{z}. \label{eq:met08}%
\end{equation}

Special values of the Kummer functions are
\begin{equation}
\label{eq:met09}M(a,a,z)=e^{z},\quad U(a,a+1,z)=z^{-a},
\end{equation}
and our asymptotic expansions reduce smoothly to these elementary values as
$b\to a$. We prefer giving expansions of $U(a,b+1,z)$ for this reason, and not
of $U(a,b,z)$, which becomes the incomplete gamma function $e^{z}%
\Gamma(1-a,z)$ as $b\to a$. For more details on the Kummer functions we refer
to \cite{Olde:2010:CHF}.

In the following two sections the asymptotic expansions have a front term of
the form $e^{z \mathcal{A}}$, where $\mathcal{A}$ can be written in terms of
\begin{equation}
\label{eq:met10}A(\mu)=\mu\left(  \tau-\ln\tau-1\right)  -\alpha\ln(1-\mu
\tau),
\end{equation}
where $\mu=\pm(\alpha-\beta)$, $\tau=t_{0}/\mu$, and $t_{0}$ is the relevant
saddle point.

\section{The expansion of {\boldmath{$M(a,b,z)$}}}

\label{sec:Mbgea} We start with $b\geq a$ and use the notation
\begin{equation}
\lambda=b-a,\quad\mu=\frac{\lambda}{z}=\beta-\alpha=\frac{b-a}{z}.
\label{eq:Mbga01}%
\end{equation}
The Kummer relation $M(a,b,z)=e^{z}M(b-a,b,-z)$ together with the integral
\begin{equation}
M(a,b,z)=\frac{\Gamma(b)}{\Gamma(a)\Gamma(b-a)}\int_{0}^{1}e^{zt}%
t^{a-1}(1-t)^{b-a-1}\,dt,\quad \Re a >0,\quad  \Re(b-a)>0, \label{eq:Mbga02}%
\end{equation}
gives
\begin{equation}%
\begin{array}
[c]{@{}r@{\;}c@{\;}l}%
M(a,b,z) & = & \displaystyle\frac{\Gamma(b)e^{z}}{\Gamma(a)\Gamma(b-a)}%
\int_{0}^{1}e^{-zt}t^{b-a-1}(1-t)^{a-1}\,dt\\[8pt]
& = & \displaystyle\frac{\Gamma(b)e^{z}}{\Gamma(a)\Gamma(\lambda)}\int_{0}%
^{1}e^{-z\phi(t)}\,\frac{dt}{t(1-t)},\label{eq:Mbga03}%
\end{array}
\end{equation}
where
\begin{equation}
\phi(t)=t-\alpha\ln(1-t)-\mu\ln t,\quad\phi^{\prime}(t)=-\frac{t^{2}%
-(\beta+1)t+\mu}{t(1-t)}. \label{eq:Mbga04}%
\end{equation}
Note that in this case $\dsp{\Phi(t)=\frac{1}{t(1-t)}}$, the  function shown in \eqref{eq:met01}. 

The saddle point $t_{0}$ inside the interval $(0,1)$ is given by
\begin{equation}
\begin{array}
[c]{@{}r@{\;}c@{\;}l}%
t_{0}&=&\dsp{\tfrac{1}{2}(\beta+1)-\tfrac{1}{2}\sqrt{(\beta+1)^{2}-4\mu}}\\[8pt]
&=&\dsp{\frac{2\mu
}{\beta+1+\sqrt{(\beta+1)^{2}-4\mu}}
=\frac{2\mu
}{\beta+1+\sqrt{(\beta-1)^{2}+4\alpha}},}
\end{array}
 \label{eq:Mbga05}%
\end{equation}
with expansion
\begin{equation}
t_{0}=\frac{\mu}{\beta+1}+\frac{\mu^{2}}{(\beta+1)^{3}}+\mathcal{O}\left(
\mu^{3}\right)  ,\quad\mu\rightarrow0. \label{eq:Mbga06}%
\end{equation}
From this, and from $\phi^\prime(t)$ given in \eqref{eq:Mbga04}, we see that the
saddle point $t_{0}$ vanishes, as $\mu\rightarrow0$.

We use the transformation given in \eqref{eq:met03} and obtain
\begin{equation}
M(a,b,z)=\frac{\Gamma(b)}{\Gamma(a)}e^{z-z\mathcal{A}}F_{\lambda}(z),\quad
F_{\lambda}(z)=\frac{1}{\Gamma(\lambda)}\int_{0}^{\infty}s^{\lambda-1}%
e^{-zs}f(s)\,ds, \label{eq:Mbga07}%
\end{equation}
where
\begin{equation}
\mathcal{A}=\phi(t_{0})-\psi(s_{0}),\quad f(s)=\frac{s}{t(1-t)}\frac{dt}%
{ds},\quad\frac{dt}{ds}=\frac{\psi^{\prime}(s)}{\phi^{\prime}(t)},\quad f(s)=\frac{s-\mu}{(\beta+1)t-t^2-\mu}.
\label{eq:Mbga08}%
\end{equation}
As in \eqref{eq:met05}, with details in the Appendix, we can obtain the expansion
\begin{equation}
M(a,b,z)\sim\frac{\Gamma(b)}{\Gamma(a)}e^{z-z\mathcal{A}}z^{-\lambda}%
\sum_{n=0}^{\infty}\frac{f_{n}(\mu)}{z^{n}},\quad z\rightarrow\infty.
\label{eq:Mbga09}%
\end{equation}

To find $f_{0}(\mu)$ we need the derivative $dt/ds$ at $s=s_{0}=\mu$. Using
l'H{\^{o}}pital's rule we have
\begin{equation}
\left.  \frac{dt}{ds}\right\vert _{s=s_{0}}=\frac{\psi^{\prime\prime}(s_{0}%
)}{\phi^{\prime\prime}(t_{0})\left.  \frac{dt}{ds}\right\vert _{s=s_{0}}%
},\quad\psi^{\prime\prime}(s_{0})=\frac{1}{\mu},\quad\phi^{\prime\prime}%
(t_{0})=\frac{\beta t_{0}^{2}-2\mu t_{0}+\mu}{t_{0}^{2}(1-t_{0})^{2}}.
\label{eq:Mbga10}%
\end{equation}
This gives
\begin{equation}
\left(  \left.  \frac{dt}{ds}\right\vert _{s=s_{0}}\right)  ^{2}=\frac
{\psi^{\prime\prime}(s_{0})}{\phi^{\prime\prime}(t_{0})}\quad\Longrightarrow
\quad f_{0}(\mu)=\sqrt{\frac{\mu}{\beta t_{0}^{2}-2\mu t_{0}+\mu}}.
\label{eq:Mbga11}%
\end{equation}

We normalize the coefficients of the expansion by writing
\begin{equation}
\label{eq:Mbga12}M(a,b,z)\sim e^{z-z\mathcal{A}}\frac{\Gamma(b)}{\Gamma
(a)}z^{a-b}f_{0}(\mu)\sum_{n=0}^{\infty}\frac{\widetilde{f}_{n}(\mu)}{z^{n}},
\quad\widetilde{f}_{n}(\mu)=\frac{f_{n}(\mu)}{f_{0}(\mu)},
\end{equation}
as $z\to\infty$ and $b\ge a$.

In applications and numerical testing it may be convenient to use a scaled
function. We write
\begin{equation}
\label{eq:Mbga13}M(a,b,z)=e^{z}\frac{\Gamma(b)}{\Gamma(a)}z^{a-b}%
{\widetilde{M}}(a,b,z),
\end{equation}
 and we summarise the above results in the following theorem.

\begin{theorem}\label{theo:theo01}
The scaled Kummer function defined in \eqref{eq:Mbga13}
has the asymptotic expansion 
\begin{equation}
\label{eq:Mbga14}{\widetilde{M}}(a,b,z) \sim e^{-z\mathcal{A}} f_{0}(\mu)
\sum_{n=0}^{\infty}\frac{\widetilde{f}_{n}(\mu)}{z^{n}}, \quad z\to \infty,
\end{equation}
uniformly with respect to $ b\ge a\ge a_0$, where $a_0$ is a fixed positive parameter,  $\mu$ is defined in \eqref{eq:Mbga01},  $\mathcal{A}=A(\mu)$ is defined in  
\eqref{eq:met10}, $f_0(\mu)$ is given in \eqref{eq:Mbga11}. The first coefficients are ${\widetilde{f}}_{0}(\mu)=1$ and
\begin{equation}
{\widetilde{f}}_{1}(\mu)=\frac{\mu\tau^{2}(\mu^{2}\tau^{5}-13\mu^{2}\tau
^{4}-\mu\tau^{4}+21\tau^{3}\mu+4\mu\tau^{2}-9\tau^{2}-2\tau-1)}{12(\mu\tau
^{2}-1)^{3}(1-\mu\tau)}, \label{eq:Mbga15}%
\end{equation}
where (see \eqref{eq:Mbga05})
\begin{equation}
\tau=\frac{t_{0}}{\mu}=\frac{2}{\beta+1+\sqrt{(\beta+1)^{2}-4\mu}}.
\label{eq:Mbga16}%
\end{equation}
 \end{theorem}
 \begin{proof}
For details of the proof we refer to \cite{Temme:1985:LTI}. The relation between $s$ and $t$ is one-to-one and analytic in a domain around the positive $s$-axis, where $f(s)$ of \eqref{eq:Mbga08} is an analytic function. The second saddle point $t_+$ that follows from \eqref{eq:Mbga05} by changing the sign in front of the square root, corresponds with a complex point $s_+$ that follows from the transformation in \eqref{eq:met03}, and this point is a singularity of $f(s)$. 

The representation of  $f_K(s)$ as a Cauchy-type integral derived in \cite[Section~25.2.1]{Temme:2015:AMI} can be useful to obtain a bound of the remainder $E_K(z,\mu)$ shown in the finite expansion \eqref{eq:app12} in the Appendix. For the current case we take $s\ge0$ and consider $f_K(s)$, which is  a linear combination of derivatives of the function $f(s)$ defined in  \eqref{eq:Mbga04}, in particular for large values positive values of $s$. In that case the relation between $s$ and $t$ following from the  transformation in \eqref{eq:met02} becomes in the current case with $\phi(t)$ as in \eqref{eq:Mbga04}, $s\sim-\alpha\ln(1-t)$ as $t\uparrow 1$. This gives for $f(s)$  the estimate
\begin{equation}
\label{eq:Mbga24} f(s)\sim \frac{s-\mu}{\alpha+(1-\beta)e^{-s/\alpha}},\quad s\to \infty,
\end{equation}
Hence,  for large $s$, $f(s)=\bigO(s)$, $f^{\prime}(s)=\bigO(1)$, and all higher derivatives are $\bigO\left(e^{-s/\alpha}\right)$. By using the recursive scheme in \eqref{eq:app12} we can conclude that a positive number $M_K$ exists  such that $\vert f_K(s)\vert\le M_K$, $s\ge0$.
 \end{proof}

\begin{remark}\label{rem:rem01}
The expansions \eqref{eq:Mbga09} and \eqref{eq:Mbga14} with coefficients $f_n(\mu)$ and $\widetilde{f}_n(\mu)$, suggest the separation of the large parameter $z $ and the uniformity parameter $\mu$ as two independent parameters.  However, $\mu$ depends on the three considered parameters $a$, $b$ and $z$, as follows from \eqref{eq:Mbga01}, although the explicit form of the scaled coefficient $\widetilde{f}_{1}(\mu)$ given in  \eqref{eq:Mbga15} (and that of the not shown higher coefficients)
 shows only two parameters $\tau$ and $\mu$. We use the current notation in order to stay close to the method described in the cited literature.
\end{remark}

 Note that ${\widetilde{f}}_{1}(0)=0$, just like all
coefficients ${\widetilde{f}}_{n}(\mu)$ with $n\geq1$ that we have computed.
Also, if $\mu\rightarrow0$, the factor $\Phi=e^{z}\frac{\Gamma(b)}{\Gamma
(a)}z^{a-b}$ becomes $e^{z}$, and the asymptotic expansion gives the correct
value $e^{z}$ when $\mu\rightarrow0$.

In numerical calculations with small values of $\mu$, we should write the front term
$e^{-z\mathcal{A}} f_{0}(\mu)$ with  $t_0$ replaced by $\mu\tau$ (see \eqref{eq:Mbga16}). By
\eqref{eq:Mbga08} and \eqref{eq:Mbga11} we have (see \eqref{eq:met10})
\begin{equation}
\label{eq:Mbga18}f_{0}(\mu)=\frac{1}{\sqrt{\beta\mu\tau^{2}-2\mu\tau+1}},
\quad\mathcal{A}=A(\mu)=\mu\left(  \tau-\ln\tau-1\right)  -\alpha\ln(1-\mu
\tau).
\end{equation}

\subsection{The case {\boldmath{$a\ge b$}}}

\label{subsec:Mageb} We use Lemma~\ref{lem:lem01} and write the function
$F_{\lambda}(z)$ defined for $\lambda>0$ in \eqref{eq:Mbga07} as
\begin{equation}
\label{eq:Mbga19}F_{\lambda}(z)=\frac{\Gamma(-\lambda+1)}{2\pi i}\int%
_{-\infty}^{(0+)} s^{\lambda-1}e^{zs}f(-s)\,ds.
\end{equation}
The right-hand side can be used as the analytic continuation of $F_{\lambda
}(z)$ with respect to $\lambda$ into the half-plane $\lambda\le0$, which gives
\begin{equation}
\label{eq:Mbga20}F_{-\lambda}(z)=\frac{\Gamma(\lambda+1)}{2\pi i}\int%
_{-\infty}^{(0+)} s^{-\lambda-1}e^{zs}f(-s)\,ds,\quad\lambda\ne-1,-2,-3,\ldots
\,.
\end{equation}
The saddle point analysis of the right-hand side of \eqref{eq:Mbga20} proceeds
as for the function $F_{\lambda}(z)$ defined in \eqref{eq:Mbga07}, and we can
use a similar procedure for integration by parts as explained in the Appendix. 
The expansions in \eqref{eq:met06} shows $(-1)^{n}$
because of the slightly different integration by parts method compared with
the one for obtaining \eqref{eq:met05}. On the other hand, the functions in
\eqref{eq:app14} have the same structure with $s$ replaced by $-s$ and $\mu$
by $-\mu$. Thus we find exactly the same expansion as in \eqref{eq:Mbga12} and \eqref{eq:Mbga14}.

\begin{corollary}\label{cor:cor01}
Using Lemma~\ref{lem:lem01} and Theorem~\ref{theo:theo01} we conclude that the asymptotic expansions in \eqref{eq:Mbga12} and \eqref{eq:Mbga14} can be used for $b\ge a\ge a_0$ as well as for
$b_0\le b\le a$, where $a_0$ and $b_0$ are fixed positive numbers.  
\end{corollary}

\subsection{The range of the three parameters}\label{sec:range}

With the upper bound of the remainder of the expansion, as stated in the proof of Theorem~\ref{theo:theo01},
global information may become available about the expansion and the range of the parameters. However, it is always useful to look at the coefficients of the expansion to get more detailed information. We discuss two points of interest in which limiting forms of the coefficients are relevant, and both points apply to the expansion in  \eqref{eq:Mbga14} as well as to that in \eqref{eq:Ubga12} for the $U$-function.

\begin{itemize}
\item  {\bf The behaviour for large values of \protectbold{\mu}.}

The initial interest to derive the expansions in this paper is validity when the parameter $\mu$ tends to zero. But we also want to find out how the expansion behaves for larger values of $\mu$.

Inspecting the coefficient ${\widetilde{f}}_{1}(\mu)$ in \eqref{eq:Mbga15} we see that for large values of $\mu$ it behaves like $\bigO(1/\mu)$, uniformly with respect to the parameter $\tau$ given in  \eqref{eq:Mbga16}. For the coefficients we derived for the numerical computations we find ${\widetilde{f}}_{k}(\mu)=\bigO(1/\mu^k)$ as $\mu \to \infty$. From this behaviour we conclude that the expansion in \eqref{eq:Mbga14} has a double asymptotic property: it is valid when $\mu$ or $z$ is large or if both are large. 

Numerical experiments confirm this property. If we take $ a=0.5$,  $b=0.7$,  $z=100$, our expansions with terms up to $n=4$ gives a result with relative error $2.0407\times10^{-11}$ when we use a test based on the Wronski relation for the Kummer functions. When we take $ a=50.5$,  $b=100.7$,  $z=1$, the Wronski test gives the error $6.51\times10^{-13}$.

\item  {\bf The behaviour for  \protectbold{t_0\to1}.}

The expansion in \eqref{eq:Mbga12} becomes useless when  $t_0\to1$. In that case the  factor $\mu\tau-1=t_0-1$  in the denominator of ${\widetilde{f}}_{1}(\mu)$ in \eqref{eq:Mbga15} tends to zero. This happens with all coefficients ${\widetilde{f}}_{k}(\mu)$ of the expansion, even so
that every ${\widetilde{f}}_{k}(\mu)$, $k\ge1$, has a factor $(\mu\tau-1)^k$ in the denominator.

By setting the numerator of $\phi^\prime(t)$ in \eqref{eq:Mbga04} equal to 0 we find that when  $t_0$ is replaced by $\rho$, where $\rho\in(0,1)$, the following linear relation between $\alpha$ and $\beta$ arises:
\begin{equation}
\label{eq:Mbga21}\alpha=\rho^2-\rho+(1-\rho)\beta.
\end{equation}
In Figure~\ref{fig:fig01} we show the shaded domain between the line $\alpha=\beta$ (for these values $t_0=0$) and the line that follows from the relation in \eqref{eq:Mbga21} (where $t_0=\rho$). Inside the coloured domain we have $0\le t_0 \le \rho$. After selecting $\rho$ we can use the condition $\alpha\ge \rho^2-\rho+(1-\rho)\beta$ on the parameters $\alpha $ and $\beta$ to use the asymptotic expansion with $t_0\le \rho$. In the figure we have taken $\rho=\frac{4}{5}$.

In the sector above the diagonal $\alpha=\beta$ we have $t_0<0$, because in that case $\mu <0$. As we have explained in the previous subsection, we can use the expansion in \eqref{eq:Mbga12}  also for $\mu <0$, and we see that for all positive values of $\alpha$ and $ \beta$ above the line governed by the relation in  \eqref{eq:Mbga21} we have $t_0\le\rho$.

The other  quantity  $\mu\tau^2-1=t_0^2/\mu-1$  in the denominator of ${\widetilde{f}}_{1}(\mu)$  vanishes  if $(\beta+1)t_0=2\mu$, which implies $(\beta-1)^2=-4\alpha$. This cannot happen if $\alpha$ and $ \beta$  are positive. In fact, when this happens two saddle points coincide, and we need Airy functions to describe the asymptotic behaviour; see \cite{Dunster:1989:UAE}.

\end{itemize}

\begin{figure}[tb]
\begin{center}
        \includegraphics[width=6.5cm]{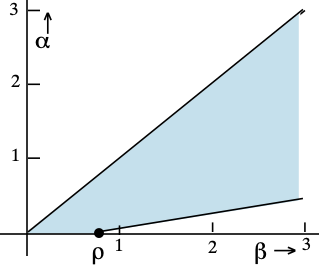} 
\end{center}
\caption{\small
In the coloured domain the saddle point $t_0$ satisfies $0\le t_0\le \rho<1$. The line from the point $(\rho,0)$ to the right has the equation given in \eqref{eq:Mbga21}. In the figure we have used $\rho=\frac45$. For further details we refer to the text.}
\label{fig:fig01}
\end{figure}

We conclude that we can use the expansions  in \eqref{eq:Mbga12} and  \eqref{eq:Ubga12} for a wide range of the parameters $a$, $b$ and $z$.

\section{The expansion of {\boldmath{$U(a,b,z)$}}}

\label{sec:Ubgea} In this section we again use the notation
\begin{equation}
\lambda=b-a,\quad\mu=\frac{\lambda}{z}=\beta-\alpha=\frac{b-a}{z},
\label{eq:Ubga01}%
\end{equation}
and we start the analysis assuming that $b\geq a$. We take the contour
integral
\begin{equation}
U(a,b,z)=\frac{\Gamma(1-a)}{2\pi i}\int_{{-\infty}}^{{(0+)}}e^{{zs}}s^{{a-1}%
}{(1-s)^{{b-a-1}}}ds,\quad\Re z>0, \label{eq:Ubga02}%
\end{equation}
where $a\neq1,2,3,\ldots$. The contour cuts the real axis between $0$ and~$1$.
At that point the fractional powers are determined by $\mathrm{ph}\,(1-s)=0$
and $\mathrm{ph}\,s=0$. We use the Kummer relation $U(a,b,z)=z^{1-b}%
U(a-b+1,2-b,z)$ and obtain
\begin{equation}
U(a,b+1,z)=\frac{z^{-b}\Gamma(\lambda+1)}{2\pi i}\int_{{-\infty}}^{{(0+)}%
}e^{zt}t^{-\lambda-1}(1-t)^{-a}\,dt. \label{eq:Ubga03}%
\end{equation}
We write this in the form
\begin{equation}
U(a,b+1,z)=\frac{z^{-b}\Gamma(\lambda+1)}{2\pi i}\int_{{-\infty}}^{{(0+)}%
}e^{z\phi(t)}\,\frac{dt}{t}, \label{eq:Ubga04}%
\end{equation}
where
\begin{equation}
\phi(t)=t-\alpha\ln(1-t)-\mu\ln t,\quad\phi^{\prime}(t)=-\frac{t^{2}%
-(\beta+1)t+\mu}{t(1-t)}. \label{eq:Ubga05}%
\end{equation}
The saddle point $t_{0}$ inside the interval $(0,1)$ is given by
\begin{equation}
t_{0}=\tfrac{1}{2}(\beta+1)-\tfrac{1}{2}\sqrt{(\beta+1)^{2}-4\mu}\,=\frac
{2\mu}{\beta+1+\sqrt{(\beta+1)^{2}-4\mu}}, \label{eq:Ubga06}%
\end{equation}
with expansion
\begin{equation}
t_{0}=\frac{\mu}{\beta+1}+\frac{\mu^{2}}{(\beta+1)^{3}}+\mathcal{O}\left(
\mu^{3}\right)  ,\quad\mu\rightarrow0. \label{eq:Ubga07}%
\end{equation}
Again, we see that the saddle point $t_{0}$ vanishes as $\mu\rightarrow0$.

We use the transformation shown in \eqref{eq:met03} and obtain
\begin{equation}
\label{eq:Ubga08}U(a,b+1,z)=z^{-b}e^{z\mathcal{A}} G_{\lambda}(z),\quad
G_{\lambda}(z)=\frac{\Gamma(\lambda+1)}{2\pi i}\int_{{-\infty}}^{{(0+)}}
s^{-\lambda-1} e^{zs}p(s)\,ds,
\end{equation}
where, with $s_{0}=\mu$,
\begin{equation}
\label{eq:Ubga09}\mathcal{A}=\phi(t_{0})-\psi(s_{0}), \quad p(s)=\frac{s}%
{t}\frac{dt}{ds},\quad\frac{dt}{ds}=\frac{\psi^{\prime}(s)}{\phi^{\prime}(t)},\quad
p(s)=\frac{(t-1)(s-\mu)}{t^2-(\beta+1)t+\mu}.
\end{equation}
As in \eqref{eq:met06} we can obtain the expansion
\begin{equation}
\label{eq:Ubga10}U(a,b+1,z)\sim e^{z\mathcal{A}} z^{-a}\sum_{n=0}^{\infty
}(-1)^{n}\frac{p_{n}(\mu)}{z^{n}},\quad z\to\infty.
\end{equation}
The first coefficient is
\begin{equation}
\label{eq:Ubga11}p_{0}(\mu)=\frac{\mu}{t_{0}}\sqrt{\frac{\psi^{\prime\prime
}(s_{0})}{\phi^{\prime\prime}(t_{0})}}=(1-t_{0})\sqrt{\frac{\mu}{\beta
t_{0}^{2}-2\mu t_{0}+\mu}}.
\end{equation}

When the parameters are large it may be convenient in numerical tests to use a
scaled function as we did for the $M$-function in \eqref{eq:Mbga13}. Here we
write
\begin{equation}
\label{eq:Ubga12}U(a,b+1,z)=z^{-a}{\widetilde{U}}(a,b+1,z)
\end{equation}
We summarise the results for the $U$-function in the following theorem.

\begin{theorem}\label{theo:theo02}
The scaled Kummer function defined in \eqref{eq:Ubga12}
has the asymptotic expansion 
\begin{equation}
\label{eq:Ubga13}{\widetilde{U}}(a,b+1,z)\sim e^{z\mathcal{A}} p_{0}(\mu)\sum_{n=0}^{\infty
}(-1)^{n}\frac{{\widetilde{p}}_{n}(\mu)}{z^{n}}, \quad
{\widetilde{p}}_{n}(\mu)=\frac{p_{n}(\mu)}{p_{0}(\mu)},\quad z\to \infty,
\end{equation}
uniformly with respect to $ b\ge a\ge a_0$, where $a_0$ is a fixed positive parameter,  
$\mu$ is defined in \eqref{eq:Ubga01},  $\mathcal{A}=A(\mu)$ is defined in  
\eqref{eq:met10}, $p_0(\mu)$ is given in \eqref{eq:Ubga11}.  The first coefficients of the expansion \eqref{eq:Ubga12} are ${\widetilde{p}}_{0}(\mu)=1$ and
\begin{equation}
\label{eq:Ubga14}\widetilde{p}_{1}(\mu)=\frac{\mu\tau^{2}(1-\tau)\left(
\mu^{2}\tau^{4} - \mu\tau^{3} + 8\mu\tau^{2} - 9\tau+ 1\right)  }{12(\mu
\tau^{2} - 1)^{3}(\mu\tau-1)},
\end{equation}
where $\tau=t_{0}/\mu$, with $t_0$ given \eqref{eq:Ubga06}. 
 \end{theorem}
\begin{proof}
The saddle point contour of the integral in \eqref{eq:Ubga08} is given by $\Im\psi(s)=\Im\psi(\mu)=0$, and is governed by \begin{equation}
\label{eq:Ubga15}
\rho=\mu\frac{\theta}{\sin\theta},\quad  s=\rho e^{i\theta}, \quad -\pi<\theta<\pi.
\end{equation}
In the $t$-plane a similar contour through the saddle point $t_0$ can be defined. On these contours  the relation between $t$ and $s$ is one-to-one, and $p(s)$ is analytic on the contour given in \eqref{eq:Ubga15}. For large $s$ and $t$ on the saddle point contours we have $s\sim t$, and from \eqref{eq:Ubga09} we conclude that $p(s)\sim1$ for large $s$ on the contour given in \eqref{eq:Ubga15}, and it can be verified that all derivatives are bounded. It follows, as in the proof of Theorem~\ref{theo:theo01}, that we can find a bound for the remainder in the finite expansion related to the expansion given in \eqref{eq:Ubga13}.
\end{proof}

For small values of $\mu$, we
write the quantities $p_{0}(\mu)$ and $A$ in terms of $\tau$. We have
\begin{equation}
\label{eq:Ubga16}p_{0}(\mu)=\frac{1-\mu\tau}{\sqrt{\beta\mu\tau^{2}-2\mu
\tau+1}}, \quad\mathcal{A}=A(\mu)=\mu\left(  \tau-\ln\tau-1\right)  -\alpha
\ln(1-\mu\tau).
\end{equation}
We see that $\mathcal{A}\to0$ and $p_{0}(\mu)\to1$ as $\mu\to0$, that is, as
$b\to a$. Also, all coefficients $\widetilde{p}_{n}(\mu)$, $n\ge1$, tend to
zero when $\mu\to0$ and in that case ${\widetilde{U}}(a,b+1,z)\to1$. This
confirms that the expansion of $U(a,b+1,z)$ tends to the elementary value
$U(a,a+1,z)=z^{-a}$ given in \eqref{eq:met09}.

\subsection{The case {\boldmath{$a\ge b$}}}
\label{subsec:Uageb} 
As in Section~\ref{subsec:Mageb} we have the following  result about the case $a\ge b$.
\begin{corollary}\label{cor:cor02}
Using Lemma~\ref{lem:lem01} and Theorem~\ref{theo:theo02} we conclude that the asymptotic expansions in \eqref{eq:Ubga10} and \eqref{eq:Ubga13} can be used for $b\ge a\ge a_0$ as well as for
$b_0\le b\le a$, where $a_0$ and $b_0$ are fixed positive numbers.  
\end{corollary}

\section{Numerical verifications}

\label{sec:num} For a detailed discussion on the computation of Kummer
functions and other hypergeometric functions we refer to the recent paper
\cite{Johansson:2019:CHF}, where arbitrary-precision implementations are
considered. Our paper focuses on asymptotic methods for the Kummer functions,
and in this section we give information on the performance of our expansions
using a limited number of terms.

To avoid comparisons by using other software, the relative errors shown in the
tables are computed by verifying recurrence relations written in the stable
forms
\begin{equation}
\label{eq:num01}\frac{zM(a+1,b+1,z)+bM(a,b,z)}{bM(a+1,b,z)}=1,\quad
\frac{aU(a+1,b,z)+U(a,b-1,z)}{U(a,b,z)}=1.
\end{equation}
When we use the scaled functions introduced in \eqref{eq:Mbga13} and
\eqref{eq:Ubga12} we can write these relations as
\begin{equation}
\label{eq:num02}\frac{z{\widetilde{M}}(a+1,b+1,z)+a{\widetilde{M}}%
(a,b,z)}{z{\widetilde{M}}(a+1,b,z)}=1,\quad\frac{a{\widetilde{U}%
}(a+1,b,z)+z{\widetilde{U}}(a,b-1,z)}{z{\widetilde{U}}(a,b,z)}=1.
\end{equation}
We can also use the relation
\begin{equation}
\label{eq:num03}aM(a,b,z)\, U(a+1,b+1,z)+\frac{a}{b}%
M(a+1,b+1,z)\,U(a,b,z)=\frac{e^{z}\Gamma(b)}{z^{b}\Gamma(a)},
\end{equation}
which follows from the Wronskian of the Kummer functions (see
\cite[Section3~13.2(vi), 13.3(ii)]{Olde:2010:CHF}). In terms of the scaled
functions we can write
\begin{equation}
\label{eq:num04}\frac{a}{z}{\widetilde{M}}(a,b,z)\,{\widetilde{U}%
}(a+1,b+1,z)+{\widetilde{M}}(a+1,b+1,z)\,{\widetilde{U}}(a,b,z)=1.
\end{equation}

We give tables showing the relative errors in the computations for a selection
of the parameters $a$, $b$ and $z$. Our computations are done with Maple
(version 2021.2), with $Digits=16$, without using the multi-precision
possibilities. We have compared our asymptotic results with Maple's $KummerM$
and $KummerU$ function codes, and found that for the $M$-functions this
comparison is reliable, but for the $U$-functions it is not (again, with
$Digits=16$). For example, when we take $a=130.0$, $b=25.1$ and $z=100.0$, Maple
gives the value $U(a,b+1,z)=-2.033\ldots\times10^{-234}$, a negative
result\footnote{With these parameter values and $Digits=32$ Maple's procedure
$KummerU$ gives $3.8723892985558665\times10^{-293}$ with $U_{asym}%
(a,b+1,z)\doteq3.872389298556251\times10^{-293}$. Compared with a computation
with Matlab, $Digits=32$, $KummerU$ gives $3.8723892985558665\times10^{-293}$,
identical with the Maple result.}. Matlab (version R2021b, $Digits=16$) gives
$U(a,b+1,z)=3.872389298555866\times10^{-293}$, and our asymptotic result gives
$U_{asym}(a,b+1,z)=3.872389298556390\times10^{-293}$.

We have also done a few other tests with Matlab (version R2021b, $Digits=16$)
for the parameter values $a=130.0$ $b=25.1$ and $z=100.0$, and we conclude
that Matlab performs better than Maple in this example, with a recursion test
on $KummerU$ based on \eqref{eq:num01} giving a relative error $0$. 


In Table \ref{tab:tab01} we give the relative errors in the computation of the
scaled functions ${\widetilde{M}}(a,b,z)$ and ${\widetilde{U}}(a,b,z)$ for
$z=500$, $b=500$, several values of $a$ by using expansions \eqref{eq:Mbga14}
and \eqref{eq:Ubga12} with terms up to $n=4$. The errors are computed by using
the scaled recurrence relations in \eqref{eq:num02}. We observe that for these
$z$, $a$ and $b$ for $n=0$ (expansions with only one term equal to~1) the
approximations give a nice estimate, and that for $n=3$ and $n=4$ the relative
errors are nearly the same.

In Table \ref{tab:tab02} we give the relative errors in the computation of the
Wronski relation \eqref{eq:num04} for the scaled functions ${\widetilde{M}%
}(a,b,z)$ and ${\widetilde{U}}(a,b,z)$ for $z=500$ and a selection values of
$a$ and $b$ by using expansion \eqref{eq:Mbga14} and \eqref{eq:Ubga12} with
terms up to $n=4$. The better values are for the smaller values of $\vert
a-b\vert$. When we repeat the computations with $z=5$ and $z=50$ with the same
values of $a$ and $b$, the relative errors match the values of this table
quite well. This indicates, as explained in \S\ref{sec:range}, that our new expansions include the previous results obtained in \cite{Temme:2021:AKH} for only large $a$ and $b$.

\renewcommand{\arraystretch}{1.2} \begin{table}[ptb]
\caption{ Relative errors in the computation of the scaled functions
${\protect\widetilde{M}}(a,b,z)$ and ${\protect\widetilde{U}}(a,b,z)$ for
$z=500$, $b=500$, several values of $a$ by using expansion\eqref{eq:Mbga14}
with terms up to $n=4$. The errors are computed by using the scaled recurrence
relations in \eqref{eq:num02}. }%
\label{tab:tab01}
\[%
\begin{array}
[c]{rcccccc}%
a\quad & n=0 & n=1 & n=2 & n=3 & n=4 & \\\hline
&  &  & {\widetilde{M}(a,b,z)} &  &  & \\\hline
99 & 0.48 \times10^{-05} & 0.43 \times10^{-08} & 0.16 \times10^{-09} & 0.76
\times10^{-12} & 0.46 \times10^{-13} & \\
199 & 0.16 \times10^{-05} & 0.12 \times10^{-08} & 0.98 \times10^{-11} & 0.10
\times10^{-14} & 0.46 \times10^{-13} & \\
299 & 0.82 \times10^{-06} & 0.55 \times10^{-09} & 0.14 \times10^{-11} & 0.10
\times10^{-12} & 0.11 \times10^{-12} & \\
399 & 0.51 \times10^{-06} & 0.30 \times10^{-09} & 0.30 \times10^{-12} & 0.16
\times10^{-14} & 0.39 \times10^{-14} & \\
499 & 0.35 \times10^{-06} & 0.19 \times10^{-09} & 0.39 \times10^{-13} & 0.10
\times10^{-14} & 0.40 \times10^{-15} & \\\hline
501 & 0.35\times10^{-06} & 0.19\times10^{-09} & 0.37\times10^{-13} &
0.10\times10^{-14} & 0.10\times10^{-14} & \\
601 & 0.26\times10^{-06} & 0.13\times10^{-09} & 0.53\times10^{-13} &
0.25\times10^{-13} & 0.25\times10^{-13} & \\
701 & 0.20\times10^{-06} & 0.89\times10^{-10} & 0.64\times10^{-13} &
0.23\times10^{-13} & 0.23\times10^{-13} & \\
801 & 0.16\times10^{-06} & 0.66\times10^{-10} & 0.74\times10^{-13} &
0.33\times10^{-13} & 0.33\times10^{-13} & \\
901 & 0.13\times10^{-06} & 0.51\times10^{-10} & 0.24\times10^{-12} &
0.20\times10^{-12} & 0.20\times10^{-12} & \\\hline
&  &  & {\widetilde{U}(a,b,z)} &  &  & \\\hline
99 & 0.29\times10^{-05} & 0.45\times10^{-08} & 0.17\times10^{-09} &
0.28\times10^{-12} & 0.39\times10^{-13} & \\
199 & 0.84\times10^{-06} & 0.15\times10^{-09} & 0.11\times10^{-10} &
0.34\times10^{-13} & 0.66\times10^{-13} & \\
299 & 0.40\times10^{-06} & 0.77\times10^{-10} & 0.18\times10^{-11} &
0.35\times10^{-13} & 0.42\times10^{-13} & \\
399 & 0.23\times10^{-06} & 0.80\times10^{-10} & 0.40\times10^{-12} &
0.10\times10^{-14} & 0.10\times10^{-14} & \\
499 & 0.15\times10^{-06} & 0.63\times10^{-10} & 0.87\times10^{-13} &
0.10\times10^{-14} & 0.00\times10^{-00} & \\\hline
501 & 0.14\times10^{-06} & 0.62\times10^{-10} & 0.85\times10^{-13} &
0.00\times10^{-00} & 0.80\times10^{-15} & \\
601 & 0.10\times10^{-06} & 0.48\times10^{-10} & 0.37\times10^{-14} &
0.10\times10^{-14} & 0.17\times10^{-14} & \\
701 & 0.73\times10^{-07} & 0.37\times10^{-10} & 0.43\times10^{-13} &
0.64\times10^{-13} & 0.64\times10^{-13} & \\
801 & 0.55\times10^{-07} & 0.28\times10^{-10} & 0.11\times10^{-13} &
0.36\times10^{-13} & 0.36\times10^{-13} & \\
901 & 0.43\times10^{-07} & 0.23\times10^{-10} & 0.16\times10^{-12} &
0.18\times10^{-12} & 0.18\times10^{-12} & \\\hline
\end{array}
\]
\end{table}\renewcommand{\arraystretch}{1.0}

\renewcommand{\arraystretch}{1.2} \begin{table}[ptb]
\caption{ Relative errors in the computation of the Wronski relation
\eqref{eq:num04} for the scaled functions ${\protect\widetilde{M}}(a,b,z)$ and
${\protect\widetilde{U}}(a,b,z)$ for $z=500$ and several values of $a$ and $b$
by using expansion \eqref{eq:Mbga14} and \eqref{eq:Ubga12} with terms up to
$n=4$. }%
\label{tab:tab02}
\[%
\begin{array}
[c]{rcccccc}%
\ a \quad & b=101 & b=301 & b=501 & b=701 & b=901 & \\\hline
101 & 0.00 \times10^{-00} & 0.46 \times10^{-12} & 0.14 \times10^{-11} & 0.42
\times10^{-12} & 0.71 \times10^{-12} & \\
301 & 0.52 \times10^{-13} & 0.40 \times10^{-15} & 0.32 \times10^{-13} & 0.73
\times10^{-13} & 0.63 \times10^{-13} & \\
501 & 0.13 \times10^{-12} & 0.15 \times10^{-13} & 0.10 \times10^{-13} & 0.50
\times10^{-14} & 0.27 \times10^{-13} & \\
701 & 0.17 \times10^{-12} & 0.89 \times10^{-13} & 0.31 \times10^{-13} & 0.00
\times10^{-00} & 0.67 \times10^{-13} & \\
901 & 0.14 \times10^{-12} & 0.14\times10^{-12} & 0.18 \times10^{-12} & 0.14
\times10^{-13} & 0.10 \times10^{-14} & \\\hline
\end{array}
\]
\end{table}\renewcommand{\arraystretch}{1.0}

To handle ratios of gamma functions with large arguments, which occur in the
expansion in \eqref{eq:Mbga12}, we can use
\begin{equation}
\label{eq:num05}\frac{\Gamma(b)}{\Gamma(a)}=e^{\ln\Gamma(b)-\ln\Gamma(a)},
\end{equation}
or the asymptotic expansions of $\Gamma(z)$ or $\ln\Gamma(z)$; see
\cite[Section~5.11]{Askey:2010:GAM}.

Another aspect that requires attention in numerical evaluations of the
asymptotic results is the factor $e^{\pm z \mathcal{A}}$ in front of the
expansions. $\mathcal{A}$ has always the form $\pm A(\pm\mu)$, with
$A(\mu)=\mu\left(  \tau-\ln\tau-1\right)  -\alpha\ln(1-\mu\tau)$ given in
\eqref{eq:met10}. Especially when $\mu$ is small (which we always allow in our
asymptotic results), the logarithmic term must be calculated accurately. For
small $x$ we have $\ln(1+x)=x+\mathcal{O}(x^{2})$, and when we first compute
$1+x$ information may get lost. For, say $-\frac12\le x\le\frac12$, we use
the relation
\begin{equation}
\label{eq:num06}\mathrm{arctanh}\,z=\tfrac12\ln\frac{1+z}{1-z} \quad
\Longrightarrow\quad\ln(1+x)=2\,\mathrm{arctanh}\,\frac{x}{2+x},
\end{equation}
and either use a power series expansion of $\mathrm{arctanh}\,z$ for, say
$-\frac13\le z\le\frac15$, or the Maple code for $\mathrm{arctanh}\,z$.

\section{Concluding remarks}

\label{sec:concl}  We derived new asymptotic expansions for positive
values of $a$, $b$ and $z$, at least one of which is large, and in this way 
we have been able to extend the results of \cite{Temme:2021:AKH},
which are only valid for large $a$ and $b$.
By proving the relation between an
integral on the positive line and a contour integral in the complex plane we
have also reduced the number of expansions from four to two. With numerical
tables we have demonstrated the performance of the expansions for a small
selection of the parameters. More extensive testing is needed to verify the
performance of the new expansions with respect of the range of the parameters
$a$ and $b$ when given a value of $z$ and the number of terms of the expansions.

\section{Appendix: Evaluating the coefficients}

\label{sec:append} In the two Sections~\ref{sec:Mbgea} and \ref{sec:Ubgea}, we
use the transformation in \eqref{eq:met03}, and the first step is to express
$t$ as a function of $s$ near $s_{0}=\mu$, the saddle point in the $s$-domain
where $\psi^{\prime}(s)=(s-\mu)/s$ vanishes. We explain the procedure
considering the transformation for Section~\ref{sec:Mbgea}, where, see
\eqref{eq:Mbga04},
\begin{equation}
\label{eq:app01}\phi(t)= t-\alpha\ln(1-t)-\mu\ln t,\quad\phi^{\prime
}(t)=-\frac{t^{2}-(\beta+1)t+\mu}{t(1-t)},
\end{equation}
with saddle point $t_{0}\in(0,1)$ given by
\begin{equation}
\label{eq:app02}t_{0}=\tfrac12(\beta+1)-\tfrac12\sqrt{(\beta+1)^{2}-4\mu
}=\frac{2\mu}{\beta+1+\sqrt{(\beta+1)^{2}-4\mu}}.
\end{equation}

To find $t$ as function of $s$ near $s_{0}$ we write the transformation in
\eqref{eq:met03} in the form of the local expansions
\begin{equation}
\sum_{k=2}^{\infty}\frac{\phi^{(k)}(t_{0})}{k!}(t-t_{0})^{k}=\sum
_{k=2}^{\infty}\frac{\psi^{(k)}(s_{0})}{k!}(s-s_{0})^{k}, \label{eq:app03}%
\end{equation}
which we can write as
\begin{equation}
(t-t_{0})\sqrt{\sum_{k=2}^{\infty}\frac{1}{k!}\phi^{(k)}(t_{0})(t-t_{0}%
)^{k-2}}=(s-s_{0})\sqrt{\sum_{k=2}^{\infty}\frac{1}{k!}\psi^{(k)}%
(s_{0})(s-s_{0})^{k-2}}, \label{eq:app04}%
\end{equation}
where the square roots are positive for positive values of $s$ and $t$. The
relation satisfies the condition imposed on the mapping in \eqref{eq:met03},
that is, $\mathrm{sign}(t-t_{0})=\mathrm{sign}(s-s_{0})$ in the present case
for $t\in(0,1)$ and $s>0$.

We substitute the expansion $\displaystyle t=t_{0}+\sum_{k=1}^{\infty}t_{k}
(s-s_{0})^{k} $ and find for the first coefficient
\begin{equation}
\label{eq:app05}t_{1}=\sqrt{\frac{\psi^{(2)}(s_{0})}{\phi^{(2)}(t_{0})}},
\end{equation}
where, again, the sign of the square root is positive to satisfy the condition
$\mathrm{sign}(s-s_{0})=\mathrm{sign}(t-t_{0})$. The other coefficients
$t_{k}$ can be found by simple computer algebra methods. The first few are
\begin{equation}
\label{eq:app06}t_{2}=\frac{\psi_{3}-\phi_{3}t_{1}^{3} }{6\phi_{2}t_{1}},
\quad t_{3}=\frac{ 5\phi_{3}^{2}t_{1}^{6} -3\phi_{2}\phi_{4}t_{1}^{6} -
4\phi_{3}\psi_{3}t_{1}^{3} + 3\phi_{2}\psi_{4}t_{1}^{2} -\psi_{3}^{2}}%
{72\phi_{2}^{2}t_{1}^{3}},
\end{equation}
where $\phi_{k}=\phi^{(k)}(t_{0})$, $\psi_{k}=\psi^{(k)}(s_{0})$, $k\ge2$.

The next step is to find the coefficients $a_{k}(\mu)$ of the expansion
$\displaystyle f(s)=\sum_{k=0}^{\infty}a_{k}(\mu)(s-s_{0})^{k}$, where in the
present example $f(s)$ is given in \eqref{eq:Mbga08}. The first coefficients
are
\begin{equation}
a_{0}(\mu)=\frac{\mu t_{1}}{t_{0}(1-t_{0})},\quad a_{1}(\mu)=\frac{2\mu
t_{0}t_{1}^{2}-2\mu t_{2}t_{0}^{2}+2\mu t_{2}t_{0}-\mu t_{1}^{2}-t_{0}%
^{2}t_{1}+t_{0}t_{1}}{t_{0}^{2}(1-t_{0})^{2}}. \label{eq:app07}%
\end{equation}

When we have enough of these coefficients, we can evaluate the coefficients
$f_{n}(\mu)$ needed in the expansions in \eqref{eq:met05} and
\eqref{eq:met06}. The first few $f_{n}(\mu)$ are
\begin{equation}
\label{eq:app08}%
\begin{array}
[c]{ll}%
f_{0}(\mu)= a_{0}(\mu),\quad f_{1}(\mu)= \mu a_{2}(\mu),\quad f_{2}(\mu)=
\mu\left(  2a_{3}(\mu)+3\mu a_{4}(\mu)\right)  , & \\[8pt]%
f_{3}(\mu)= \mu\left(  6a_{4}(\mu)+20\mu a_{5}(\mu)+15\mu^{2}a_{6}%
(\mu)\right)  , & \\[8pt]%
f_{4}(\mu)= \mu\left(  24a_{5}(\mu)+130\mu a_{6}(\mu)+210\mu^{2}a_{7}%
(\mu)+105\mu^{3}a_{8}(\mu)\right)  . &
\end{array}
\end{equation}

\begin{remark}
\label{rem:rem02} The reviewer of \cite{Temme:2021:AKH} observed: {\em It seems that the numerical coefficients are the same as the sequence A269940 in the OEIS. It would be worth investigating this in the future.}  
See also https://oeis.org/A269940. This would imply
\begin{equation}
f_{n}(\mu)=\sum_{k=1}^{n}T(n,k)\mu^{k}a_{k+n}(\mu),\quad T(n,k)=\sum_{m=0}%
^{k}(-1)^{n+m+k}\binom{n+k}{n+m}s(n+m,m), \label{eq:app09}%
\end{equation}
where $s(n,m)$ are the Stirling numbers of the first kind. 
We have proved this relation by using mathematical induction.
\end{remark}

Observe that, to avoid square roots in the formulas, we do not substitute
$t_{0}$ given in \eqref{eq:app02} and $t_{1}$ in \eqref{eq:app05}. A second
point is to reduce the number of variables in the formulas. We see that the
given scaled coefficient ${\widetilde{f}}_{1}(\mu)$ given in \eqref{eq:Mbga15}
is a function of two parameters only, namely $\mu$ and $\tau=t_{0}/\mu$. The
derivatives of $\phi(t)$ given in \eqref{eq:app01} at $t_{0}$ are functions of
$t_{0}$, $\mu$ and $\beta$, but we have used the relation $(\beta
+1)t_{0}=t_{0}^{2}+\mu$ (see the numerator of $\phi^{\prime}(t)$ in
\eqref{eq:app01}) to eliminate $\beta$ from the formulas. This way we can make
the final coefficients as simple as possible.

A final point is to have stable representations. When we would use
${\widetilde{f}}_{1}(\mu)$ given in \eqref{eq:Mbga15} with $\tau$ replaced by
$t_{0}/\mu$, a form arises that is still analytic at $\mu=0$ (a crucial value
in our asymptotics), but from a numerical point of view it becomes undefined
at $\mu=0$. The representation of ${\widetilde{f}}_{1}(\mu)$ in \eqref{eq:Mbga15}  in terms of the parameter $\tau$ is stable for small values of $\mu$, just as the higher coefficients.

\subsection{The integration by parts procedure}

\label{subsec:ip} To explain the relation between the coefficients $f_{n}%
(\mu)$ and $a_{n}(\mu)$ as shown in \eqref{eq:app08} we give a few steps in
the integration by parts procedure. We write $f(s)$ of the integral in
\eqref{eq:Mbga07} as\ $f(s)=\bigl(f(s)-f(\mu)\bigr)+f(\mu)$. Then we have
\begin{equation}
\label{eq:app10}%
\begin{array}
[c]{@{}r@{\;}c@{\;}l@{}}%
F_{\lambda}(z) & = & \displaystyle z^{-\lambda}f(\mu)- \frac1{z\Gamma
(\lambda)}\int_{0}^{\infty}\frac{f(s)-f(\mu)}{s-\mu}\,de^{-z\psi(s)}\\[8pt]
& = & \displaystyle z^{-\lambda}f(\mu)+\frac1{z\Gamma(\lambda)}\int%
_{0}^{\infty}s^{\lambda-1}e^{-zs} f_{1}(s)\,ds,
\end{array}
\end{equation}
where
\begin{equation}
\label{eq:app11}f_{1}(s)=s\frac{d}{ds}\frac{f(s)-f(\mu)}{s-\mu}.
\end{equation}
Continuing this procedure we obtain for $K=0,1,2,\ldots$
\begin{equation}
\label{eq:app12}%
\begin{array}
[c]{@{}r@{\;}c@{\;}l@{}}%
\displaystyle z^{\lambda}\,F_{\lambda}(z) & = & \displaystyle \sum_{k=0}^{K-
1}\frac{f_{k}(\mu)}{z^{k}}+ \frac1{z^{K}}E_{K}(z,\mu),\\[8pt]%
\displaystyle f_{k}(s) & = & \displaystyle s\frac d{ds}\frac{f_{k-1}%
(s)-f_{k-1}(\mu)}{s-\mu},\quad k=1,2,\ldots,\quad f_{0}(s)=f(s),\\[8pt]%
\displaystyle E_{K}(z,\mu) & = & \displaystyle \frac1{\Gamma(\lambda)}\int%
_{0}^{\infty}s^{\lambda-1} e^{-zs}f_{K}(s)\,ds.
\end{array}
\end{equation}
Eventually we obtain the complete asymptotic expansion
\begin{equation}
\label{eq:app13}F_{\lambda}(z)\sim z^{-\lambda}\sum_{n=0}^{\infty}\frac
{f_{n}(\mu)}{z^{n}}.
\end{equation}

As shown in \eqref{eq:app08}, the coefficients $f_{n}(\mu)$ can be expressed
in terms of the coefficients $a_{n}(\mu)$. To verify this we write
\begin{equation}
\label{eq:app14}f(s)=\sum_{m=0}^{\infty}a_{m}(\mu)(s-\mu)^{m}, \quad
f_{n}(s)=\sum_{m=0}^{\infty}c_{m}^{(n)}(s-\mu)^{m}, \quad c_{m}^{(0)}%
=a_{m}(\mu).
\end{equation}
We see that $f_{n}(\mu)=c_{0}^{(n)}$ and we have from \eqref{eq:app12}
\begin{equation}
\label{eq:app15}f_{n+1}(s)=\sum_{m=0}^{\infty}c_{m}^{(n+1)}(s-\mu)^{m}%
=s\sum_{m=1}^{\infty}c_{m}^{(n)}(m-1)(s-\mu)^{m-2}.
\end{equation}
This gives the recursion
\begin{equation}
\label{eq:app16}c_{m}^{(n+1)}=mc_{m+1}^{(n)}+\mu(m+1)c_{m+2}^{(n)},\quad
m,n=0,1,2,\ldots\,.
\end{equation}
All these coefficients can be expressed in terms of $a_{m}(\mu)$, and
especially $c_{0}^{(n)}=f_{n}(\mu)$. This gives the relations between
$f_{n}(\mu)$ and $a_{n}(\mu)$, the first ones being given in \eqref{eq:app08}.

The procedure described here can be applied in exactly the same way to obtain
the coefficients $p_{n}(\mu)$ of the asymptotic expansion in
\eqref{eq:Ubga10}. A limited number of coefficients are provided in this
article, but more of these are available from the authors.

\section*{Acknowledgments}
The authors thank the reviewers for carefully reading the manuscript and for constructive suggestions.\\
NMT acknowledges financial support from {\emph{Ministerio de Ciencia e
Innovaci\'on}}, project MTM2012-11686. \newline NMT thanks CWI, Amsterdam, and
the Universidad de Cantabria, Spain, for support.\newline EJMV thanks TUDelft,
Delft, for support.

\end{document}